\documentclass[12pt,reqno]{amsart}
\usepackage{amsmath,amsfonts,amsthm,color,amssymb, mathrsfs}

\usepackage{pxfonts}
\makeatletter
\@tfor\@tempa:=
\alpha \beta \gamma \delta \epsilon \zeta \eta \theta \iota \kappa \lambda \mu
\nu \xi \pi \rho \sigma \tau \upsilon \phi \chi \psi \omega \varepsilon
\vartheta \varpi \varrho \varsigma \varphi  
\do {% for some reason, must proceed in two steps
	\count255 \numexpr\@tempa+"7000\relax
	\expandafter\mathchardef\@tempa \count255 }
\makeatother

\usepackage{xcolor}
\usepackage[T1]{fontenc}
\usepackage{wasysym}
\usepackage[normalem]{ulem}
\usepackage{stmaryrd} 

\usepackage[left=1in, right=1in, top=1.1in,bottom=1.1in]{geometry}
\usepackage{enumitem}
\usepackage{hyperref}
\hypersetup{
	colorlinks   = true,
	citecolor    = blue,
	linkcolor=blue
}
\allowdisplaybreaks
\usepackage{csquotes}
\usepackage{graphicx}
\usepackage{pstricks}
\usepackage{lmodern}
\newtheorem{theorem}{Theorem}[section]

\newtheorem{definition}{Definition}

\newtheorem{proposition}{Proposition}

\newtheorem{remark}{Remark}

\def\rme{{\rme}}

\def\la{\left\langle}
\def\ra{\right\rangle}

\def\R{\mathbb R}
\def\E{\mathbb E}

\def\bR{\mathbb R}
\def\bN{\mathbb N}
\def\bE{\mathbb E}
\def\bD{\mathbb D}

\def\cH{\mathcal H}
\def\cF{\mathcal F}
\def\cS{\mathcal S}
\def\Dom{\mathrm {Dom}}

\newcounter{bean}
\newcommand{\benuma}{\setlength{\labelwidth}{.25in}
	
	\begin{list}
		{(\alph{bean})}{\usecounter{bean}}}
	\newcommand{\eenuma}{\end{list}}

\begin{document}
	\title[Stochastic fractional heat equation]{On a class of stochastic fractional heat equations}
	\author[J. Song]{Jian Song}
	\address{Research Center for   Mathematics and Interdisciplinary Sciences, Shandong University, China}
	\email{txjsong@sdu.edu.cn}
	
	\author[M. Wang]{Meng Wang}
	\address{School of Mathematics, Shandong University, Jinan, Shandong, China}
	\email{wangmeng22@mail.sdu.edu.cn}
	
	\author[W. Yuan]{Wangjun Yuan}
	\address{University of Luxembourg, Department of Mathematics, Maison du Nombre 6, Avenue de la Fonte L-4364 Esch-sur-Alzette. Luxembourg.}
	\email{ywangjun@connect.hku.hk}
	\vspace{-2cm}
	\maketitle
	\vspace{-1cm}
	\begin{abstract}
		For the fractional heat equation $\frac{\partial}{\partial t} u(t,x) = -(-\Delta)^{\frac{\alpha}{2}}u(t,x)+  u(t,x)\dot W(t,x)$ where  the covariance function of the Gaussian noise $\dot W$ is defined by the heat kernel, we establish Feynman-Kac formulae for both  Stratonovich and Skorohod solutions, along with their respective moments. In particular, we prove that  $d<2+\alpha$ is a sufficient and necessary condition for the equation to have a unique square-integrable mild Skorohod solution.  One motivation lies in the occurrence of this equation in the study of a random walk in  random environment which is generated by a field of independent random walks starting from a Poisson field. 
		
		\medskip\noindent\textbf{Keywords.} Stochastic fractional heat equation; Stratonovich solution; Skorohod solution; Malliavin calculus; Feynman-Kac formula; directed polymer model \smallskip
		
		\noindent\textbf{AMS 2020 Subject Classifications.} 60H15, 35R60, 60G60
	\end{abstract}
	%\end{adjustwidth}
	\section{introduction}
	Consider
	the following stochastic fractional heat equation 
	\begin{equation} \label{eq-SPDE}
		\left\{
		\begin{aligned}
			\dfrac{\partial}{\partial t} u(t,x) =&\, -(-\Delta  )^{\frac{\alpha}{2}}u(t,x)+  u(t,x)\dot W(t,x),\,\,t\geq0,x\in\R^d,\\
			u(0,x)=&\,u_0(x),
		\end{aligned}
		\right.
	\end{equation}
	where $-(-\Delta)^{\frac{\alpha}{2}}$ is the fractional Laplacian with $\alpha\in(0,2]$,  the initial value $u_0$ is a bounded continuous function,  the noise  $\dot W$ is a generalized space-time Gaussian field with covariance function given by
	\begin{align}\label{eq-covariance}
		\bE \left[ \dot W(t,x) \dot W(s,y) \right] = p_{|t-s|}(x-y),
	\end{align}
	where $p_t(x)= (2\pi t)^{-\frac{d}{2}} \exp (-\frac{|x|^2}{2t})$  for $t>0$ and $x\in\R^d$ is the heat kernel, and the product in $u\dot W$ is either an ordinary product or a Wick product.

	In this note, we shall prove Feynman-Kac formula for the Stratonovich solution of \eqref{eq-SPDE} (corresponding to the ordinary product in $u\dot W$) which has been done for the case $\alpha=2, d=1$ in \cite{SSSX2021},  and prove the existence and uniqueness of the Skorohod solution (corresponding to the Wick product in $u\dot W$)  and derive Feynman-Kac formulae for the solution and its moments.  We obtain the Feynman-Kac formulae for the Stratonovich solution of \eqref{eq-SPDE}  and its moments for $d<2$  (i.e. $d=1$ and $\alpha\in(0,2]$, see Theorem \ref{thm:FK}), prove that the Skorohod solution exists uniquely  for $d<2+\alpha$ (see Theorem \ref{Thm-solution}  and Proposition \ref{prop:sko}), and obtain the Feynman-Kac formulae for the Skorohod solution and its moments for $d<2$ and for $d<2+\alpha$ respectively (see Theorem~\ref{thm:fk-sko-u} and Theorem~\ref{thm:fk-sko-mom}). Moreover, we  show that $d<2+\alpha$ is also a necessary condition for \eqref{eq-SPDE} to have  a unique square-integrable mild Skorohod solution (see Proposition \ref{prop:sko}).

	We present a selection of related works within the SPDEs literature, recognizing that this list is not exhaustive.  It was shown in \cite{h02} that \eqref{eq-SPDE} with  space-time white noise $\dot W$  admits a unique square integrable Skorohod solution  only when $d=1$.  When the noise $\dot W$ is white in space and  colored in time,  the condition for the existence of a unique Skorohod solution of \eqref{eq-SPDE}  was identified in \cite{hn09}. In \cite {hns11} the Feynman-Kac formulae for the Stratonovich and Skorohod solutions as well as for the moments were obtained, where $\dot W$ is the partial derivative of   fractional Brownian sheet with the Hurst parameters  within $(\frac 12,1)$. The case that $\dot W$ is a general Gaussian noise was studied in  \cite{hhnt}.  All the above-mentioned works considered the case $\alpha=2$.  The result in \cite{hns11} was extended in \cite{chs12} to the case $\alpha\in(0,2]$, and to SPDEs with the differential operator  associated with a symmetric L\'evy process in \cite{S17}. For a similar model as the one in \cite{S17} with the noise white in time and colored in space, the moment formulae were established in \cite{conus13}. We also refer to 
	\cite{Walsh86,cm94, fk13, davar14, bc14,  chsx15, bc16, swy23} and the references therein for more related results.

	The stochastic heat equation with multiplicative Gaussian noise in the form of \eqref{eq-SPDE} is intimately connected with  the directed polymer model in random environment, of which the study was initiated in~\cite{AKQ} and further developed in \cite{CSZ17a,rang2020,SSSX2021,CG2023,rsw23}, etc. It was proved in~\cite{AKQ}  that for a  simple symmetric random walk in i.i.d random environment on $\mathbb N\times \mathbb Z$, the rescaled  partition function  converges weakly to the It\^o (i.e. Skorohod) solution of  \eqref{eq-SPDE} with  $d=1, \alpha=2$ and  $\dot W$ being  space-time white noise; this result was extended from a simple random walk  to a long-range random walk in \cite{CSZ17a} and correspondingly $\alpha$ belongs to $(1,2]$. When the environment is correlated in space but still independent in time, the case of simple random walk (resp. long-range random walk) was considered in \cite{rang2020}
	(resp. \cite{CG2023}), in which they proved  that the rescaled partition function converges weakly to the It\^o solution of \eqref{eq-SPDE} with $\alpha=2$ (resp. $\alpha\in(1,2]$) and  the noise being white in time and colored in space.  When the environment is correlated in time but independent in space and the random walk is possibly long-range, it was shown in \cite{rsw23} that there are two types of rescaled partition functions which converge weakly to the Skorohod solution and the Stratonovich solution respectively of \eqref{eq-SPDE} with $d=1,\alpha\in(1,2]$ and the Gaussian noise $\dot W$ being colored in time and white in space.  For a simple symmetric random walk in time-space correlated random environment generated by a Poisson field of independent random walks, the rescaled partition function was proved in \cite{SSSX2021}  to converge weakly to the Stratonovich solution of \eqref{eq-SPDE} with $d=1, \alpha=2$ and  the covariance of $\dot W$ being given by \eqref{eq-covariance}.  
	
	The study of the equation \eqref{eq-SPDE} with the covariance of the noise given by \eqref{eq-covariance} is inspired by the above-mentioned works, in particular by \cite{SSSX2021} and \cite{rsw23}.  
	In light of \cite{SSSX2021} and  \cite{rsw23}, our result suggests that for a long-range random walk in a Poisson field of independent random walks,  there should exist two  types of partition functions converging weakly  to the Stratonovich solution for $d=1$ and to the Skorohod solution for $d<2+\alpha$ respectively, as done in \cite{rsw23} (see Theorem 1.2 and Theorem~1.3 therein). 
	
	The rest of the paper is organized as follows. In section \ref{sec:pre}, some preliminary knowledge is provided. In Sections \ref{sec:Str} and \ref{sec:Skr}, the Stratonovich solution and the Skorohod solution are studied respectively.

	% We conjecture that the partition functions in the directed polymer consisting of a long-range random walk and the same random environment defined by Poisson field of random walks as described in \cite{SSSX2021} converge to the solution of \eqref{eq-SPDE} in the sense of both Stratonovich and Skorohod. This is a future work we intend to undertake. 

	\section{Preliminaries} \label{sec:pre}
	In this section, we provide some facts that will be used in this note on $\alpha$-stale process and Gaussian analysis. 
	
	Let $X=\{X_t,t\ge 0\}$ be a $d$-dimensional $\alpha$-stable process associated with the fractional Laplacian $-(-\Delta)^{\frac{\alpha}2}$, which is independent of $\dot W$. Denote the transition density   function of $X$  by $g_{\alpha}(t,x)$.   Note that
	when $\alpha=2$, $g_\alpha(t,x)=p_{2t}(x)$.
	The Fourier transform of $g_{\alpha}(t,\cdot)$ with respect to the spatial variables is given by 
	\begin{align}\label{eq-Fourier-p}
		\cF g_\alpha(t,\xi) :=\int_{\R^d}\exp\big(-ix\cdot \xi \big)g_{\alpha}(t,x)dx= \exp \left( -  t |\xi|^\alpha\right).
	\end{align}
	Moreover, for $p>0$,
	\begin{align}\label{eq-integral-p}
		\int_{\R^d}\left| \cF g_\alpha(t,\xi) \right|^p d\xi=(t p)^{-\frac{d}{\alpha}}\int_{\R^d}e^{-|\xi|^\alpha}d\xi = Cp^{-\frac{d}{\alpha}} t^{-\frac{d}{\alpha}}
	\end{align}
	for some constant $C$ that depends on $\alpha$ and $d$.

	%{\red We need $d\le \alpha$, which force $d=1$ and $\alpha \in (1,2]$.}
	
	Let $\cH$ be the completion of the space of $C_c^\infty([0,\infty) \times \bR^d)$ of smooth functions with compact support equipped with the inner product
	\begin{align} \label{eq-H norm}
		\langle f,g \rangle_{\cH}
		=& \int_{\bR_+^2} \int_{\bR^{2d}} f(s,x) g(t,y) p_{|t-s|}(x-y) dxdy dsdt \nonumber \\
		=& \int_{\bR_+^2 } \int_{\bR^{d}} \cF f(s,\cdot) (\xi) \overline{\cF g(t,\cdot) (\xi)} \cF p_{|t-s|}(\xi) d\xi dsdt.
	\end{align}
	We denote by $\|\cdot\|_{\cH}$ the norm induced by the inner product.
	
	In a complete probability space $(\Omega,\mathcal F,P)$, define an isonormal Gaussian process $\{W(g), g \in \cH \}$  with covariance
	$   \bE \left[ W(g) W(h) \right] = \langle g,h \rangle_{\cH}$
	for all $g,h \in \cH$. In this paper, we also denote 
	\[W(g):=\int_{\R_+}\int_{\R^d}g(t,x)W(dt,dx).\]
	In particular, if $g(s,y) = {\bf 1}_A (s,y)$  where $A$ is of the form $[0,t] \times \prod_{j=1}^d [0,x_j]$, we also write  $W(t,x) = W({\bf 1}_A)$. The Gaussian noise $\dot W(t,x)$ can be identified  as the generalized derivative $\frac{\partial^{1+d}}{\partial t \partial x_1\cdots\partial x_d} W(t,x)$.

	Let  $C_p^{\infty}(\bR_+ \times \bR^d)$ be the set of all infinitely continuously differentiable functions with all partial derivatives being polynomial growth. We define the set
	\begin{align*}
		\cS = \left\{ f \left( W(g_1), \ldots, W(g_n) \right): n \in \bN_+, f \in C_p^{\infty}(\bR_+ \times \bR^d), g_1, \ldots, g_n \in \cH \right\}.
	\end{align*}
	Then we can define the Malliavin derivative $D$ on $\cS$ by
	\begin{align*}
		D f \left( W(g_1), \ldots, W(g_n) \right) \
		= \sum_{j=1}^n \dfrac{\partial f}{\partial x_j} \left( W(g_1), \ldots, W(g_n) \right) g_j.
	\end{align*}
	The operator $D$ is closable from $L^2(\Omega)$ to $L^2(\Omega;\cH)$ and we denote by $\mathbb D^{1,2}$ the closure of $\mathcal S$ under the norm 
	$$\|F\|_{1,2} = \left(\bE \left[ |F|^2 \right] +  \bE \big[ \left\| D F \right\|^2_{\cH} \big] \right)^{\frac12}.$$
	Define
	\begin{align*}
		\Dom (\mathfrak{\delta}) := \left\{ u: \left| \bE \left[ \langle DF, u \rangle_{\cH} \right] \right| \le c_u \|F\|_2, \text{ for all } F \in \bD^{1,2} \right\}.
	\end{align*}
	For $u \in \Dom(\mathfrak{\delta})$, the divergence operator (also called Skorohod integral) $\mathfrak{\delta}(u)$ defined by the following duality formula $\bE \left[ F \mathfrak{\delta}(u) \right] = \bE \left[ \langle DF, u \rangle_{\cH} \right].$
	We also write 
	\[\mathfrak{\delta}(u):= \int_{\R_+} \int_{\R^d} u(s,x)W(ds,dx).\]  
	
	Recall the following formula, which will be used in the subsequent sections,
	\begin{align}\label{formula}
		FW(\phi)=\mathfrak{\delta}(F\phi)+\la DF,\phi\ra_{\mathcal H},
	\end{align}
	for any $\phi\in \mathcal H$ and $F\in \mathbb D^{1,2}$.
	For more details, we refer to \cite{n06}.

	Throughout this note, we use $C$ to denote a generic constant which may vary in different contexts.

	\section{Stratonovich solution}\label{sec:Str}
	In this section, we will study the solution of \eqref{eq-SPDE} in the Stratonovich sense. Denote 
	\begin{align}\label{def-A}
		A_{t,x}^{\varepsilon,\delta}(r,y)=\int_0^t\psi_\delta(t-s-r)p_\varepsilon(X_s^x-y)ds,
	\end{align}
	where \[\psi_\delta(t)=\frac{1}{\delta}I_{[0,\delta]}(t),\, t\ge0,\] and $X^x_s=X_s+x$. Hence $A_{t,x}^{\varepsilon,\delta}(r,y)$ is an approximation of  $\boldsymbol{\delta}(X_{t-r}^x-y)$ for small  $\varepsilon$ and $\delta$, where $\boldsymbol{\delta}(\cdot)$ is the Dirac delta function.  It was shown in  \cite[Proposition~2.3]{SSSX2021} that $A_{t,x}^{\varepsilon,\delta}$ belongs to $\cH$ almost surely for all $\varepsilon, \delta>0$ and $$I^{\varepsilon,\delta}_{t,x}=\int_0^t\int_{\R^d}A_{t,x}^{\varepsilon,\delta}(r,y)W(dr,dy)=W(A_{t,x}^{\varepsilon,\delta})$$ is well-defined. By the same argument used in the proof of \cite[Proposition~2.3]{SSSX2021}, we can get the following result. 
	\begin{theorem}\label{convergence}
		Assume $d=1$. Then we have 
		$I_{t,x}^{\varepsilon,\delta}$ converges in $L^2$ as $(\varepsilon,\delta)\to0$ to a limit $I_{t,x}$  denoted by
		\begin{align}\label{I_{t,x}}
			I_{t,x}=:\int_0^t\int_{\R^d}\boldsymbol{\delta}(X_{t-r}^x-y)W(dr,dy)=W(\boldsymbol{\delta}(X_{t-\cdot}^x-\cdot)).
		\end{align}
		Conditional on $X$, $I_{t,x}$ is  a Gaussian random variable with mean $0$ and variance
		\begin{align}\label{cov}
			\text{Var}[I_{t,x}|X]=\int_0^t\int_0^tp_{|s-r|}(X_s-X_r) drds.
		\end{align}
	\end{theorem}
	\begin{proof}
		For $\varepsilon, \delta, \varepsilon', \delta'>0$,  we have that by \eqref{eq-H norm},
		\begin{align}\label{eq-AA}
			\la A_{t,x}^{\varepsilon,\delta},A_{t,x}^{\varepsilon',\delta'}\ra_{\mathcal H}&=\int_{[0,t]^4}\int_{\R^{2d}}\psi_{\delta}(t-s-u)p_\varepsilon(X_s^x-y)p_{|u-v|}(y-z)\\&\hspace{5em}\times \psi_{\delta'}(t-r-v)p_{\varepsilon'}(X_r^x-z)dydzdudvdrds.	\nonumber
		\end{align}
		By using the fact that $p_t(x)\leq (2\pi t)^{-\frac 12}$ and $p_\varepsilon,p_{\varepsilon'}$  are probability density functions, we have 
		\begin{equation}
			\begin{aligned}
				\left\langle A_{t,x}^{\varepsilon,\delta}, A_{t,x}^{\varepsilon',\delta'} \right\rangle_{\cH}
				& \leq \int_{[0,t]^4} \psi_{\delta}(t-s-u) \psi_{\delta'}(t-r-v) (2\pi |u-v|)^{-\frac 12} dudvdrds\\
				&\leq C\int_0^t\int_0^t|s-r|^{-\frac 12}drds<\infty,
			\end{aligned}
		\end{equation}
		where the second inequality follows from Lemma A.3 of \cite{hns11}.
		Note that
		\[\lim_{\varepsilon,\delta,\varepsilon',\delta'\to0} \left\langle A_{t,x}^{\varepsilon,\delta}, A_{t,x}^{\varepsilon',\delta'} \right\rangle_{\cH}
		=  \int_0^t\int_0^t p_{|s-r|}(X_s-X_r)drds. \]
		By the dominated convergence theorem, we can deduce that
		\begin{align*}
			\lim_{\varepsilon,\delta,\varepsilon',\delta'\to0}\E\Big[\left\langle A_{t,x}^{\varepsilon,\delta}, A_{t,x}^{\varepsilon',\delta'} \right\rangle_{\cH}\Big]=\E\Big[\int_0^t\int_0^t p_{|s-r|}(X_s-X_r)drds\Big].
		\end{align*}
		Hence, $\{I_{t,x}^{\varepsilon,\delta}\}_{\varepsilon,\delta>0}$ is a Cauchy sequence in $L^2(\Omega)$. Then, \eqref{cov} holds since
		\begin{align*}
			\lim_{\varepsilon,\delta\to0}\E\Big[\left\langle A_{t,x}^{\varepsilon,\delta}, A_{t,x}^{\varepsilon,\delta} \right\rangle_{\cH}\Big]=\E\Big[\int_0^t\int_0^t p_{|s-r|}(X_s-X_r)drds\Big].
		\end{align*}
	\end{proof}
	
	\begin{theorem}\label{thm:exp}
		The following estimate holds if and only if $d=1$:  \begin{align}\label{exp-integrability}
			\E\left[\exp\left(\int_0^t\int_0^tp_{|s-r|}(X_s-X_r) drds\right)\right]<\infty.
		\end{align}
	\end{theorem}
	\begin{proof}  When $d=1$, 
		clearly \eqref{exp-integrability} follows from
		\begin{align*}
			\int_0^t\int_0^tp_{|s-r|}(X_s-X_r) drds\leq \int_0^t\int_0^t(2\pi|s-r|)^{-\frac{1}{2}}drds<\infty.
		\end{align*}
		Now we show the necessity of the condition $d=1$, for which we show that $d=1$ is a necessary condition for $\E\left[\int_0^t\int_0^tp_{|r-s|}(X_s-X_r) drds\right]<\infty.$ Indeed, we have
		\begin{align*}
			&\E\left[\int_0^t\int_0^tp_{|s-r|}(X_s-X_r) drds\right]= 2\int_0^t \int_0^s\int_{\R^d} p_{s-r}(y) g_\alpha(s-r, y) dy drds\\
			&=2\int_0^t \int_0^s\int_{\R^d} e^{-(s-r) (\tfrac 12|\xi|^2+|\xi|^\alpha)} d\xi drds\ge 2\int_0^t \int_0^s\int_{|\xi|\ge1} e^{-C(s-r) (|\xi|^2+|\xi|^\alpha)} d\xi drds\\
			&\ge 2\int_0^t \int_0^s\int_{|\xi|\ge1} e^{-2C(s-r) |\xi|^2} d\xi drds. 
		\end{align*}
		Noting that \[ 2\int_0^t \int_0^s\int_{\R^d} e^{-2C(s-r) |\xi|^2} d\xi drds = C\int_0^t\int_0^s (s-r)^{-\frac d2}drds \] is finite iff $d=1$ and that $2\int_0^t \int_0^s\int_{|\xi|< 1} e^{-2C(s-r) |\xi|^2} d\xi drds$ is finite for all $d$,  we deduce that $d=1$ is a necessary condition for $\E\left[\int_0^t\int_0^tp_{|s-r|}(X_s-X_r) drds\right]<\infty$ and hence for \eqref{exp-integrability}.
	\end{proof}
	Consider the following  approximation of $\dot W(t,x)$:
	\begin{align*}
		\dot W^{\varepsilon,\delta}(t,x)
		= \int_0^t\int_{\R^d} \psi_{\delta}(t-s) p_{\varepsilon}(x-y) W(ds,dy),
		% = W\big( \phi^{\varepsilon,\delta}_{t,x} \big). 
	\end{align*}
	The following definition of Stratonovich integral is borrowed from \cite[Definition 4.1]{hns11}.
	\begin{definition}\label{Stratonovich}
		Given a random field $v=\{v(t,x),t\geq 0,x\in\R^d\}$ such that 
		$\int_0^T\int_{\R^d}|v(t,x)|dxdt<\infty
		$
		almost surely for all $T>0$, the Stratonovich integral $\int_0^T\int_{\R^d}v(t,x)W(dt,dx)$ is defined as the following limit in probability,	\begin{align*}\lim\limits_{\varepsilon,\delta\to 0} \int_0^T\int_{\R^d}v(t,x)\dot W^{\varepsilon,\delta}(t,x)dxdt.
		\end{align*}
	\end{definition}
	
	Let $\mathcal F_t^W$ be the $\sigma$-algebra generated by $\{W(s,x), s\in[0,t], x\in\R^d\}.$
	
	\begin{definition}\label{def:mild-Stra}   A random field $\{u(t,x),t\ge0, x\in\R^d\}$ is a mild Stratonovich solution to~\eqref{eq-SPDE} if for all $t\ge 0$ and $x\in\R^d$,  $u(t,x)$ is  $\mathcal F_t^W$-measurable and the following integral equation holds
		\begin{equation}\label{eq-midl-stra}
			u(t,x) = (g_{\alpha}(t, \cdot)*u_0)(x) +\int_0^t \int_{\R^d} g_\alpha(t-s,x-y) u(s,y) W(ds,dy),  
		\end{equation}
		where \[(g_{\alpha}(t,\cdot)*u_0)(x)=\int_{\R^d}g_{\alpha}(t,x-y)u_0(y)dy\] and the stochastic integral is in the Stratonovich sense of Definition \eqref{Stratonovich}.
	\end{definition}
	Below is the main result for mild Stratonovich solution. 
	\begin{theorem}\label{thm:FK}
		Assume $d=1$. Then,\begin{align}\label{u}
			u(t,x)=\E_X\left[u_0(X_t^x)\exp\left(\int_0^t\int_{\R}\boldsymbol{\delta}(X_{t-r}^x-y)W(dr,dy)\right)\right]
		\end{align}
		is a mild Stratonovich solution of \eqref{eq-SPDE}. Furthermore, for any positive integer $p$,
		\begin{align}\label{e:fk-mom-str}
			\bE \left[u(t,x) ^p \right]
			= \bE \bigg[ \prod_{j=1}^p u_0 (X_t^{(j)}+x) \exp \bigg( \frac 12 \sum\limits_{j,k=1}^p \int_0^t\int_0^t p_{|s-r|} \left( X^{(j)}_{s}-X^{(k)}_{r} \right) dr ds\bigg) \bigg],
		\end{align}
		where $X^{(1)},\cdots,X^{(p)}$ are independent copies of $X$.
	\end{theorem}
	\begin{proof}
		Noting Theorem \ref{thm:exp}, the proof of \eqref{u} follows exactly the same argument used in \cite[Theorem 4.6]{S17} (see also the proof of \cite[Proposition 1.7]{SSSX2021}).  We sketch the proof below for the reader's convenience. 
		
		Consider the following approximation equation
		\begin{equation} \label{eq-SPDE-appro1}
			\left\{
			\begin{aligned}
				\dfrac{\partial}{\partial t} u^{\varepsilon,\delta}(t,x) =&\, -(-\Delta  )^{\frac{\alpha}{2}}u^{\varepsilon,\delta}(t,x) + u^{\varepsilon,\delta} (t,x) \dot W^{\varepsilon,\delta}(t,x),\\
				u^{\varepsilon,\delta}(0,x)=&\,u_0(x).
			\end{aligned}
			\right.
		\end{equation}
		By the classical Feynman-Kac formula, we know that
		\begin{align*}
			u^{\varepsilon,\delta}(t,x)=&\E_X\Big[u_0(X_t^x)\exp\Big(\int_0^t\dot W^{\varepsilon,\delta}(t-s,X_s^x)ds\Big)\Big]\\=&\E_X\Big[u_0(X_t^x)\exp\Big(W(A_{t,x}^{\varepsilon,\delta})\Big)\Big]=\E_X\Big[u_0(X_t^x)\exp\big(I_{t,x}^{\varepsilon,\delta}\big)\Big],
		\end{align*}
		is a mild Stratonovich solution of \eqref{eq-SPDE-appro1}.
		It follows from Theorem \ref{convergence} that $\exp\big(I_{t,x}^{\varepsilon,\delta}\big)$ converges to $\exp\big(I_{t,x}\big)$ in probability. Note that for any $\lambda \in \mathbb R$,
		\begin{align*}
			&\E\Big[\exp\big(I_{t,x}^{\varepsilon,\delta}\big)\Big]=\E\Big[\exp\Big(\tfrac 12 \lambda^2 \|A_{t,x}^{\varepsilon,\delta}\|_{\mathcal H}^2\Big)\Big]\\&\leq \E\Big[\exp\Big(\tfrac 12 \lambda^2 \int_0^t\int_0^t|s-r|^{-\frac 12}drds\Big)\Big]<\infty.
		\end{align*}
		Hence, we deduce that for any $p>0$, $u^{\varepsilon}(t,x)$ converges to $u(t,x)$ for any $t>0$ and $x\in \R$ as $\varepsilon,\delta$ tend to zero in $L^p$. 
		According to Definition \ref{def:mild-Stra},
		\begin{align}\label{mild solution appro1}
			u^{\varepsilon,\delta}(t,x) = (g_{\alpha}(t,\cdot)*u_0)(x) + \int_0^t \int_{\bR} g_\alpha(t-s,x-y) u^{\varepsilon,\delta}(s,y) \dot W^{\varepsilon,\delta}(s,y)dsdy.
		\end{align}
		To prove the desired result of this theorem, it suffices to show that
		%both sides of \eqref{def:mild-Stra} converge to those of \eqref{eq-midl-stra} in probability as $\varepsilon,\delta\to 0$, respectively. 
		\begin{align*}
			\Phi^{\varepsilon,\delta}:=\int_0^t\int_{\R}g_\alpha(t-s,x-y) \big(u^{\varepsilon,\delta}(s,y)-u(s,y)\big) \dot W^{\varepsilon,\delta}(s,y)dsdy
		\end{align*}
		converges to $0$ in probability as $\varepsilon,\delta\to 0$. Applying the formula \eqref{formula} to $\big(u^{\varepsilon,\delta}(s,y)-u(s,y)\big) W\big(\varphi_{s,y}^{\varepsilon,\delta}(r,z)\big)$  with $\varphi_{s,y}^{\varepsilon,\delta}(r,z)=\psi_{\delta}(s-r)p_{\varepsilon}(y-z)$ yields that
		\begin{align*}
			\Phi^{\varepsilon,\delta}=&\mathfrak{\delta}\Big(\int_0^t\int_{\R}g_\alpha(t-s,x-y) \big(u^{\varepsilon,\delta}(s,y)-u(s,y)\big)\varphi_{s,y}^{\varepsilon,\delta}(\cdot,\cdot) dyds\Big)\\&+\int_0^t\int_{\R}g_\alpha(t-s,x-y)\la D\big(u^{\varepsilon,\delta}(s,y)-u(s,y)\big),\varphi_{s,y}^{\varepsilon,\delta}\ra_{\mathcal H}dyds.
		\end{align*}
		
		By using relevant tools from Malliavin calculus, we can prove that $\Phi^{\varepsilon,\delta}$ converges to $0$ in $L^1$ as $\varepsilon,\delta\to 0$. We refer to \cite[Theorem 4.6]{S17} (see also the proof of \cite[Proposition 1.7]{SSSX2021}) for more details.

		Finally, the formula \eqref{e:fk-mom-str} is a direct consequence of \eqref{u}. 
	\end{proof}

	\section{Skorohod solution}\label{sec:Skr}
	In this section, we provide a sufficient and necessary condition for the existence and uniqueness of a square-integrable mild Skorohod solution of \eqref{eq-SPDE},  and prove Feynman-Kac formulae for the solution and its moments. 
	
	% The method is inspired, for instance, by \cite{S17}.
	
	\subsection{The existence and uniqueness of the Skorohod solution}
	
	\begin{definition}\label{mild solution}
		A random field $u = \{ u(t,x), t \ge 0, x \in \bR^d \}$ is a mild Skorohod solution of  \eqref{eq-SPDE}  if for all $t \ge 0$ and $x \in \bR^d$, $u(t,x)$ is $\mathcal F_t^W$-measurable and satisfies the following integral equation:
		\
		\begin{align}\label{mild solution u}
			u(t,x) = (g_{\alpha}(t,\cdot)*u_0)(x) +  \int_0^t \int_{\bR^d} g_\alpha(t-s,x-y) u(s,y) W(ds,dy),
		\end{align}
		where  the stochastic integral is in the Skorohod sense.
	\end{definition}
	
	If \eqref{eq-SPDE} admits a solution $u$, then by iteration, one can formally write
	\begin{align} \label{eq-series}
		u(t,x) = (g_{\alpha}(t,\cdot)*u_0)(x) + \sum_{n=1}^\infty I_n(\tilde f_n),
	\end{align}
	where the function $f_n$ is given by
	%\begin{align*}
	% f_n(x_1,s_1, \ldots, x_n,s_n,x,t) = \prod_{j=1}^n p_{s_{j+1}-s_j} (x_{j+1}-x_j) u_0(x_1) {\bf 1}_{0<s_1<\ldots<s_n<t},
	%\end{align*}
	\begin{align*}
		f_n(x_1,s_1, \ldots, x_n,s_n,x,t) = \prod_{j=1}^n g_\alpha( s_{j+1}-s_j,x_{j+1}-x_j) (g_{\alpha}(s_1,\cdot)*u_0)(x_1) {\bf 1}_{\{0<s_1<\ldots<s_n<t\}},
	\end{align*}
	with convention $s_{n+1} = t$ and $x_{n+1} = x$, $\tilde f_n$ is the symmetrization of $f_n$, i.e.
	\begin{align*}
		\tilde f_n(x_1,s_1, \ldots, x_n,s_n,x,t)
		=&\dfrac{1}{n!} f_n(x_{\tau(1)},s_{\tau(1)}, \ldots, x_{\tau(n)},s_{\tau(n)},x,t),
	\end{align*}
	where $\tau $ is the permutation such that $0<s_{\tau(1)}<\ldots<s_{\tau(n)}<t$
	and $I_n(\cdot)$ is the multiple Wiener integral given by
	\begin{align*}
		I_n(f_n) = \int_{\bR_+^n} \int_{\bR^{nd}} f_n(x_1,s_1, \ldots, x_n,s_n,x,t) W(ds_1,dx_1) \ldots W(ds_n,dx_n).
	\end{align*}
	Noting the expansion \eqref{eq-series} and the uniqueness of the Wiener chaos expansion, we have that the existence of a  unique 
	\emph{square-integrable} mild Skorohod solution to \eqref{eq-SPDE} is equivalent to
	\begin{align} \label{ineq-series-converge}
		\bE \left[\big|(g_{\alpha}(t,\cdot)*u_0)(x)\big|^2 \right] + \sum_{n=1}^{\infty} n! \left\| \tilde f_n(\cdot,t,x) \right\|_{\cH^{\otimes n}}^2 < \infty,
	\end{align}
	for all $(t,x) \in [0,T] \times \bR^d$. 
	
	In the following result, we show that $d<2+\alpha$ is a sufficient condition for the existence and uniqueness of a square-integrable mild Skorohod solution.
	
	\begin{theorem} \label{Thm-solution}
		Assume $d<2+\alpha$. Then the condition  \eqref{ineq-series-converge} is satisfied and hence  equation \eqref{eq-SPDE} admits a unique square-integrable  mild Skorohod solution.
	\end{theorem}
	
	\begin{proof} Without loss of generality, we assume that $u_0 (x)\equiv 1$.
		For $n \ge 1$, by \eqref{eq-H norm}, we have
		\begin{align}\label{eq-convergence for series}
			& n! \left\| \tilde f_n(\cdot,t,x) \right\|_{\cH^{\otimes n}}^2\nonumber\\ 
			&= \frac{1}{n!} \int_{[0,t]^{2n}} \int_ {\bR^{2nd}} f_n(x_{\tau(1)},s_{\tau(1)}, \ldots, x_{\tau(n)},s_{\tau(n)},x,t) \prod_{j=1}^n  p_{|s_{j}-r_{j}|}(x_{j}-y_{j} )\nonumber\\
			&\hspace{6em}\times f_n(y_{\sigma(1)},r_{\sigma(1)}, \ldots, y_{\sigma(n)},r_{\sigma(n)},x,t)   d\mathbf x d\mathbf y d\mathbf r d\mathbf s\nonumber \\
			&=\frac{1}{n!} \int_{[0,t]^{2n}}\int_{\bR^{nd}} \cF f_n(\cdot,s_{\tau(1)}, \ldots, \cdot,s_{\tau(n)},x,t) (\xi_1,\ldots,\xi_n) \prod_{j=1}^n \cF p_{|s_{j}-r_{j}|}(\xi_j)  \\
			&\hspace{6em} \times \overline{\cF f_n(\cdot,r_{\sigma(1)}, \ldots, \cdot,r_{\sigma(n)},x,t) (\xi_1,\ldots,\xi_n)} d\boldsymbol{\xi}d\mathbf r d\mathbf s\nonumber
			\\ &\leq \frac{1}{n!} \int_{[0,t]^{2n}}\int_{\bR^{nd}} \prod_{j=1}^n\exp\left(-\big[(s_{\tau(j+1)}-s_{\tau(j)})+(r_{\sigma(j+1)}-r_{\sigma(j)})\big]|\xi_j|^\alpha\right)\nonumber\\
			&\hspace{6em} \times\prod_{j=1}^n\exp\big({-\tfrac 12|s_{j}-r_{j}||\xi_j|^2}\big)d\boldsymbol{\xi}d\mathbf r d\mathbf s\nonumber,
		\end{align}
		with $0<s_{\tau(1)}<\ldots<s_{\tau(n)}<t$ and $0<r_{\sigma(1)}<\ldots<r_{\sigma(n)}<t$.  The last inequality in \eqref{eq-convergence for series} follows from \eqref{eq-Fourier-p}, maximum principle in \cite[Lemma 4.9]{S17}  and the fact that
		\begin{align*}
			\cF f_n(\cdot,s_{\tau(1)}, \ldots, \cdot,s_{\tau(n)},x,t) (\xi_1,\ldots,\xi_n)
			= \exp \left(-ix \cdot\sum_{j=1}^n \xi_j \right) \prod_{j=1}^n \cF g_\alpha(s_{\tau(j+1)}-s_{\tau(j)}, \xi_1+\cdots+\xi_j).
		\end{align*}
		Let $p, q>0$  such that $\frac{2}{p}+\frac{1}{q}$=1. Then by utilizing   H\"older inequality and \eqref{eq-integral-p}, we have that for some constant $C$ depending on $p,q,\alpha,d$, 
		\begin{align}
			& n! \left\| \tilde f_n(\cdot,t,x) \right\|_{\cH^{\otimes n}}^2\nonumber\\
			&  \leq\frac{1}{n!} \int_{[0,t]^{2n}}\bigg(\int_{\bR^{nd}} \prod_{j=1}^n\exp\left(- p(s_{\tau(j+1)}-s_{\tau(j)})|\xi_j|^\alpha\right)d\boldsymbol{\xi}\bigg)^{\frac{1}{p}}\bigg(\int_{\bR^{nd}} \prod_{j=1}^n\exp\left(-\tfrac 12q|s_{j}-r_{j}||\xi_j|^2\right)d\boldsymbol{\xi}\bigg)^{\frac{1}{q}}
			\nonumber\\&\hspace{4em}\times \bigg(\int_{\bR^{nd}} \prod_{j=1}^n\exp\left(-p(r_{\sigma(j+1)}-r_{\sigma(j)})|\xi_j|^\alpha\right)d\boldsymbol{\xi}\bigg)^{\frac{1}{p}}d\mathbf r d\mathbf s\nonumber \\
			&=\frac{1}{n!}C^n \int_{[0,t]^{2n}} \prod_{j=1}^n(s_{\tau(j+1)}-s_{\tau(j)})^{-\frac{d}{p\alpha}}\prod_{j=1}^n(r_{\sigma(j+1)}-r_{\sigma(j)})^{-\frac{d}{p\alpha}}\prod_{j=1}^n|s_{j}-r_{j}|^{-\frac{d}{2q}}d\mathbf rd\mathbf s\label{eq-H-norm}\\
			& \le\frac{1}{n!}C^n\bigg(\int_{[0,t]^{n}} \prod_{j=1}^n\left(s_{\tau(j+1)}-s_{\tau(j)}\right)^{-\frac{d}{p\alpha}\times \frac{2}{(2-\frac{d}{2q})}}d\mathbf s\bigg)^{2-\frac{d}{2q}}\label{last inequality}\\
			&=\frac{1}{n!}C^n(n!)^{2-\frac{d}{2q}}\bigg(\int_{\{0<s_1<\cdots<s_n<t\}} \prod_{j=1}^n\left(s_{j+1}-s_{j}\right)^{-\frac{d}{p\alpha}\times \frac{2}{(2-\frac{d}{2q})}}d\mathbf s\bigg)^{2-\frac{d}{2q}}\label{second equality}\\
			& = C^n(n!)^{1-\frac{d}{2q}}{ t^{n-\frac{2nd}{p\alpha(2-\frac{d}{2q})}}}\frac{\left(\Gamma(1-\frac{2d}{p\alpha(2-\frac{d}{2q})})\right)^{n(2-\frac{d}{2q})}}{\left(\Gamma(n-\frac{2nd}{p\alpha(2-\frac{d}{2q})}+1)\right)^{2-\frac{d}{2q}}}.\label{last equality}
		\end{align}
		
		In the computation above, we assume $\alpha, d, p,q$ with $\frac2p +\frac1q =1$
		satisfy the following three conditions:
		\begin{align*}
			&-\frac{d}{2q}\in (-1,0)\Longleftrightarrow d<2q,\\
			&
			\frac{2d}{p\alpha(2-\frac{d}{2q})}<1 \Longleftrightarrow d<\frac{4pq\alpha}{4q+p\alpha },\\ 
			&\Big(1-\frac{2d}{p\alpha(2-\frac{d}{2q})}\Big)(2-\frac{d}{2q})>1-\frac{d}{2q} \Longleftrightarrow d<\frac{p\alpha}{2}.
		\end{align*}
		where the first condition is due to  Lemma 5.1.1 in \cite{n06} by which we deduce \eqref{last inequality},  the second one  ensures  the integral in \eqref{second equality} is finite,  and the third one  is used  to obtain the convergence of the series.
		It turns out that  if we choose $p = (4+2\alpha)/\alpha$ and  $q=1+\frac{\alpha}{2}$,  those three conditions are satisfied, and we get the optimal condition $d<2+\alpha$ which is assumed in this theorem. The last equality \eqref{last equality} follows from Lemma C.3 of \cite{rsw23},  from which we  deduce \eqref{ineq-series-converge}  by Stirling's formula and the condition $d<2+\alpha.$ 
		The proof is complete.
	\end{proof}
	
	The following proposition indicates that the condition $d<2+\alpha$ in Theorem \ref{Thm-solution} is also necessary.  
	\begin{proposition}\label{prop:sko}
		The condition \eqref{ineq-series-converge} holds only if $ d<2+\alpha$.    
	\end{proposition}
	%\begin{remark}
	%  This proposition together with Theorem \ref{Thm-solution} indicates that, when $\alpha=2$, the condition $d<2+\alpha=4$ is a necessary and sufficient condition for the existence and uniqueness of a square-integrable mild Skorohod solution to \eqref{eq-SPDE}.
	%\end{remark}
	
	\begin{proof} Assume $u_0(x)\equiv1$.  Consider the second moment of the  first chaos,
		\begin{align*}
			\|\tilde f_1(\cdot, t,x)\|_{\mathcal H}^2=&2 \int_{\{0<r<s<t\}} \int_{\R^{2d}}p_{s-r}(y-z) g_\alpha(s,y)g_\alpha(r,z)  dydz drds \\
			=& 2\int_{\{0<r<s<t\}} \int_{\R^d} \exp\left(-\frac 12(s-r)|\xi|^2 -(s+r)|\xi|^\alpha    \right) d\xi dr ds.
		\end{align*}
		% If $\alpha=2$, we have
		% \begin{align*}
		%     \|\tilde f_1(\cdot, t,x)\|_{\mathcal H}^2
		%     =& 2\int_{\{0<r<s<t\}} \int_{\R^d} \exp\left(-\frac 12(3s+r)|\xi|^2   \right) d\xi dr ds\\
		%     \ge & 2\int_{\{0<r<s<t\}} \int_{\R^d} \exp\left(-2s|\xi|^2   \right) d\xi dr ds\\
		%     =& C \int_{\{0<r<s<t\}} s^{-\frac d 2} drds
		% \end{align*}
		% which is finite iff $d<4$.

		Taking the change of variables $u=r, v=s-r$, and choosing $M>3^{\frac{1}{2-\alpha}}$ if $\alpha<2$, and $M>1$ if $\alpha=2$, we have
		\begin{align*}
			&\int_{\{0<r<s<t\}} \int_{\{|\xi|\ge M\}} \exp\left(-\frac 12(s-r)|\xi|^2 -(s+r)|\xi|^\alpha    \right) d\xi dr ds\\&=\int_{\left\{\substack{u,v>0\\0<u+v<t}\right\}}\int_{\{|\xi|\ge M\}}\exp\left(-\frac 12v|\xi|^2-(v+2u)|\xi|^{\alpha}\right)d\xi drds\\&=\int_{\{|\xi|\ge M\}}\int_0^t\exp\left(-v(\tfrac 12 |\xi|^2+|\xi|^{\alpha})\right)dv\int_0^{t-v}
			\exp\left(-2u|\xi|^{\alpha}\right) du d\xi\\
			&=\int_{\{|\xi|\ge M\}}\frac{1}{2|\xi|^\alpha}\left\{\frac{1-\exp\left(-t(\frac 12 |\xi|^2+|\xi|^\alpha)\right)}{\frac12 |\xi|^2+|\xi|^\alpha}-\frac{\exp\left(-2t|\xi|^{\alpha}\right)-\exp\left(-t(\frac 12|\xi|^2+|\xi|^\alpha)\right)}{\frac12 |\xi|^2-|\xi|^\alpha}\right\}d\xi
		\end{align*}
		which is finite iff $d<2+\alpha$, noting that for $\alpha\in(0,2]$,
		\[\left\{\frac{1-\exp\left(-t(\frac 12 |\xi|^2+|\xi|^\alpha)\right)}{\frac12 |\xi|^2+|\xi|^\alpha}-\frac{\exp\left(-2t|\xi|^{\alpha}\right)-\exp\left(-t(\frac 12|\xi|^2+|\xi|^\alpha)\right)}{\frac12 |\xi|^2-|\xi|^\alpha}\right\} \sim \frac{C}{|\xi|^2}\]
		as $|\xi|\to \infty$,  for some $C\in(0,\infty)$. 
		
		On the other hand, we have  \[2\int_{\{0<r<s<t\}} \int_{\{|\xi|<M\}} \exp\left(-\frac 12(s-r)|\xi|^2 -(s+r)|\xi|^\alpha    \right) d\xi dr ds\] is finite for any $M\in(0,\infty)$ and thus   the desired result follows. 
	\end{proof}
	
	\begin{remark} Assume $u_0(x)\equiv1$. For the second moment of the $n$-th chaos, a direct calculation shows that 
		\[n!\|\tilde f_n(\cdot, t,x)\|_{\mathcal H^{\otimes n}}^2 =\frac1{n!} \E\left[\left(\int_0^t \int_0^t p_{|s-r|}(X_s-\tilde X_r)drds\right)^n \right],\]
		where $\tilde X$ is an independent copy of $X$, and consequently we have  \begin{equation}\label{e:FK-2mom}
			\E[u(t,x)^2]=\sum_{n=0}^\infty n!\|\tilde f_n(\cdot, t,x)\|_{\mathcal H^{\otimes n}}^2 =\E\left[\exp\left(\int_0^t \int_0^t p_{|s-r|}(X_s-\tilde X_r)drds\right)\right]. 
		\end{equation}
		Thus, we have $d<2+\alpha$ is a sufficient and necessary condition for 
		\[\E\left[\exp\left(\int_0^t \int_0^t p_{|s-r|}(X_s-\tilde X_r)drds\right)\right]<\infty.\]
		Recall that $\E\left[\exp\left(\int_0^t \int_0^t p_{|s-r|}(X_s-X_r)drds\right)\right]<\infty$ iff $d=1$ by   Theorem \ref{thm:exp}.
	\end{remark}
	
	\begin{remark}
		When the covariance function of the noise $\dot W$ is given by 
		\[ \E[\dot W(t,x) \dot W(s,y) ]=|t-s|^{2H-2}|x-y|^{-\beta}\] with $H\in(\frac12,1)$ and $\beta\in(0,d)$,   the necessary and sufficient condition $\beta<\alpha$ for the corresponding fractional SHE to have a unique square-integrable Skorohod solution is obtained in \cite{bc14} by using the large deviation result of $\int_0^t\int_0^t|X_s-\tilde X_r|^{-\beta}drds$, and the precise long-term asymptotics for the moments of the solution was obtained in \cite{chsx15, CHSS18}, by using the Feynman-Kac type large deviation result for the 
		Hamiltonian $\int_0^t\int_0^t|s-r|^{2H-2}|X_s-\tilde X_r|^{-\beta}drds$ established therein. 
	\end{remark}

	\subsection{Feynman-Kac formulae}
	
	In this subsection, we present Feynman-Kac type representations for the Skorohod solution and its moments. 
	\begin{theorem}\label{thm:fk-sko-u}
		When $d=1$, the process 
		\begin{align}\label{e:FK-skorohod}
			u(t,x)=\E_X\left[u_0(X_t^x)\exp\left(\int_0^t\int_0^t\boldsymbol{\delta}(X_{t-r}^x-y)W(dr,dy)-\frac12 \int_0^t\int_0^tp_{|s-r|}(X_s-X_r)drds\right)\right]
		\end{align}
		is the unique mild Skorohod solution to \eqref{eq-SPDE}.
	\end{theorem}
	\begin{proof}
		Note that when $d=1$, Theorem \ref{thm:FK} holds. Then the proof follows directly from the argument used in the proof of \cite[Theorem 7.2]{hns11}.
	\end{proof}
	Theorem \ref{Thm-solution} indicates that the existence and uniqueness of the Skorohod solution hold under the condition $d<2+\alpha$. However, the Feynman-Kac formula \eqref{e:FK-skorohod} for the Skorohod solution of equation \eqref{eq-SPDE} is valid only when $d=1$, as referenced in \cite[ Proposition 3.2]{hns11}. Nevertheless, we can investigate the Feynman-Kac formula for the $p$-th order moments of the Skorohod solution in the case $d<2+\alpha$.

	\begin{theorem}\label{thm:fk-sko-mom}
		Assume $d<2+\alpha$ and $\{u(t,x),t\ge 0,x\in \R^d\}$ is a mild Skorohod solution of \eqref{eq-SPDE}. Then for all positive integer $p$, we have 
		\begin{equation}
			\label{eq-FK-moment}
			\bE \left[u(t,x) ^p \right]
			= \bE \bigg[ \prod_{j=1}^p u_0 (X_t^{(j)}+x) \exp \bigg( \sum_{1 \le j<k \le p} \int_0^t\int_0^tp_{|s-r|} \left( X^{(j)}_{s}-X^{(k)}_{r} \right)dr ds \bigg) \bigg],
		\end{equation}
		where $X^{(1)},\cdots,X^{(p)}$ are independent copies of $X$.
	\end{theorem}
	
	\begin{remark}
		The formula for the second moment of the solution can be obtained by a direct calculation of the sum of the second moments of chaos (see eq.  \eqref{e:FK-2mom}).
	\end{remark}

	\begin{proof}
		When $d=1$, the desired formula \eqref{eq-FK-moment} follows from \eqref{e:FK-skorohod} directly. For $d\in(1, 2+\alpha)$,  
		we consider the following approximation of \eqref{eq-SPDE}:
		\begin{equation} \label{eq-SPDE-appro}
			\left\{
			\begin{aligned}
				\dfrac{\partial}{\partial t} u^{\varepsilon,\delta}(t,x) =&\, -(-\Delta  )^{\frac{\alpha}{2}}u^{\varepsilon,\delta}(t,x) + u^{\varepsilon,\delta} (t,x) \dot W^{\varepsilon,\delta}(t,x),\\
				u^{\varepsilon,\delta}(0,x)=&\,u_0(x).
			\end{aligned}
			\right.
		\end{equation}
		In accordance with Definition \ref{mild solution}, the mild solution of \eqref{eq-SPDE-appro} satisfies
		\begin{align}\label{mild solution appro}
			u^{\varepsilon,\delta}(t,x) = (g_{\alpha}(t,\cdot)*u_0)(x) + \int_0^t \int_{\bR^d} g_\alpha(t-s,x-y) u^{\varepsilon,\delta}(s,y) \dot W^{\varepsilon,\delta}(s,y)dyds.
		\end{align}
		Using a similar argument as in the proof of \cite[Theorem 5.6]{S17}, we have
		\begin{align} \label{eq-FK-appro}
			u^{\varepsilon,\delta}(t,x)
			= \bE_{X} \bigg[ u_0 (X_t^x) \exp \Big(  W \big( A_{t,x}^{\varepsilon,\delta}\big) - \big\| A_{t,x}^{\varepsilon,\delta} \big\|_{\cH}^2 \Big) \bigg].
		\end{align}
		satisfies \eqref{mild solution appro} and hence is a mild solution of equation \eqref{eq-SPDE-appro}.

		For the $p$-th moment of $u^{\varepsilon,\delta}(t,x)$, we have
		\begin{align} \label{eq-pth moment}
			\bE \Big[ \big(u^{\varepsilon,\delta}(t,x)\big)^p \Big]
			=& \bE \bigg[ \prod_{j=1}^p u_0 \big(X_t^{(j)}+x\big) \exp \bigg( \sum_{j=1}^p W \big( A_{t,x}^{\varepsilon,\delta,(j)} \big) -\frac12 \sum_{j=1}^p \big\| A_{t,x}^{\varepsilon,\delta,(j)} \big\|_{\cH}^2 \bigg) \bigg] \nonumber \\
			=& \bE \bigg[ \bE_{W} \bigg[ \prod_{j=1}^p u_0 \big(X_t^{(j)}+x\big) \exp \bigg( W\Big( \sum_{j=1}^p A_{t,x}^{\varepsilon,\delta,(j)} \Big) - \frac12\sum_{j=1}^p \big\| A_{t,x}^{\varepsilon,\delta,(j)} \big\|_{\cH}^2 \bigg) \bigg] \bigg] \nonumber \\
			=& \bE \bigg[ \prod_{j=1}^p u_0 \big(X_t^{(j)}+x\big) \exp \bigg( \dfrac{1}{2} \Big\|  \sum_{j=1}^p A_{t,x}^{\varepsilon,\delta,(j)} \Big\|_{\cH}^2 - \dfrac{1}{2} \sum_{j=1}^p \big\| A_{t,x}^{\varepsilon,\delta,(j)} \big\|_{\cH}^2 \bigg) \bigg] \nonumber \\
			=& \bE \bigg[ \prod_{j=1}^p u_0 \big(X_t^{(j)}+x\big) \exp \bigg(  \sum_{1 \le j<k\le p} \left\langle A_{t,x}^{\varepsilon,\delta,(j)}, A_{t,x}^{\varepsilon,\delta,(k)} \right\rangle_{\cH} \bigg) \bigg],
		\end{align}
		where $A_{t,x}^{\varepsilon,\delta,(j)}$ is defined by \eqref{def-A} with $X$ replaced by $X^{(j)}$.

		By \eqref{def-A} and \eqref{eq-H norm}, we have
		\begin{align} \label{eq-inner product}
			\left\langle A_{t,x}^{\varepsilon,\delta,(j)}, A_{t,x}^{\varepsilon,\delta,(k)} \right\rangle_{\cH} =& \int_{[0,t]^4}  \int_{\bR^{2d}}  \psi_{\delta}(t-s-u) \psi_{\delta}(t-r-v) p_{|u-v|}(y-z)\\&\hspace{3em}\times p_{\varepsilon}(X^{(j),x}_s-y) p_{\varepsilon}(X^{(k),x}_r-z) dydzdrdsdudv,\nonumber
		\end{align}
		where $X^{(j),x}_s=X_s^{(j)}+x$. By the semigroup property of the heat kernel, we have \begin{align*} \int_{\bR^{2d}}p_{\varepsilon}(X^{(j),x}_s-y) p_{\varepsilon}(X^{(k),x}_r-z) p_{|u-v|}(y-z) dydz 
			= p_{|u-v|+2\varepsilon}(X_s^{(j)}-X_r^{(k)}),
		\end{align*}
		and hence
		\begin{equation}
			\label{e:integral}
			\begin{aligned}
				\left\langle A_{t,x}^{\varepsilon,\delta,(j)}, A_{t,x}^{\varepsilon,\delta,(k)} \right\rangle_{\cH}
				= & \int_{[0,t]^4} \psi_{\delta}(t-s-u) \psi_{\delta}(t-r-v) p_{|u-v|+2\varepsilon}(X_s^{(j)}-X_r^{(k)})drds dudv\\
				= & \int_{[0,t]^4} \psi_{\delta}(u-s) \psi_{\delta}(v-r) p_{|u-v|+2\varepsilon}(X_s^{(j)}-X_r^{(k)})drds dudv,
			\end{aligned}
		\end{equation}
		Therefore, noting 
		\[\lim_{\varepsilon,\delta\to0} \left\langle A_{t,x}^{\varepsilon,\delta,(j)}, A_{t,x}^{\varepsilon,\delta,(k)} \right\rangle_{\cH}
		=  \int_{[0,t]^2} p_{|s-r|}(X_s^{(j)}-X_r^{(k)})drds, \]
		to prove the desired result \eqref{eq-FK-moment} by the dominated convergence theorem, it suffices to show, for $j\neq k$, 
		\begin{align} \label{ineq:exp-integral}  \sup_{\varepsilon,\delta>0}\sup_{t\in[0,T]} \bE \left[ \exp \left( \lambda \Big<A_{t,x}^{\varepsilon,\delta,(j)}, A_{t,x}^{\varepsilon,\delta,(k)} \Big>_{\cH} \right) \right] < \infty, \quad \text{for all} \,\,\lambda > 0.
		\end{align}
		
		By \eqref{e:integral} and the independence between $X^{(j)}$ and $X^{(k)}$ for $j\neq k$, we have
		\begin{align*}
			&\frac1{n!} \E\left[\left(\Big<A_{t,x}^{\varepsilon,\delta,(j)}, A_{t,x}^{\varepsilon,\delta,(k)} \Big>_{\cH}\right)^n\right]
			\\
			& = \frac{1}{n!} \int_{[0,t]^{4n}} \int_ {\bR^{2nd}} f_n(x_{\tau(1)},s_{\tau(1)}, \ldots, x_{\tau(n)},s_{\tau(n)},x,t) \prod_{i=1}^n p_{|u_i-v_i|+2\varepsilon}(x_i-y_i) \psi_\delta(u_i-s_i) \\
			&\hspace{6em}\times \prod_{i=1}^n\psi_\delta(v_i-r_i)f_n(y_{\sigma(1)},t_{\sigma(1)}, \ldots, y_{\sigma(n)},t_{\sigma(n)},x,t)   d\mathbf x d\mathbf y d\mathbf r d\mathbf s  d\mathbf u d \mathbf v \\
			&=\frac{1}{n!} \int_{[0,t]^{4n}} \prod_{i=1}^n \psi_\delta(u_i-s_i)\psi_\delta(v_i-r_i)\int_{\bR^{nd}} \cF f_n(\cdot,r_{\tau(1)}, \ldots, \cdot,r_{\tau(n)},x,t) (\xi_1,\ldots,\xi_n)  \\
			&\hspace{6em} \times \prod_{i=1}^n \cF p_{|u_i-v_i|+2\varepsilon}(\xi_i) \overline{\cF f_n(\cdot,s_{\sigma(1)}, \ldots, \cdot,s_{\sigma(n)},x,t) (\xi_1,\ldots,\xi_n)} d\boldsymbol{\xi}d\mathbf rd\mathbf s  d\mathbf u d \mathbf v
			\\&\leq \frac{1}{n!} \int_{[0,t]^{4n}}\prod_{i=1}^n \psi_\delta(u_i-s_i)\psi_\delta(v_i-r_i) \int_{\bR^{nd}} \prod_{i=1}^n\exp\big({-|u_i-v_i||\xi_i|^2}\big)\\
			&\hspace{6em}  \times \prod_{i=1}^n\exp\left(-c_\alpha\big[(s_{\tau(i+1)}-s_{\tau(i)})+(r_{\sigma(i+1)}-r_{\sigma(i)})\big]|\xi_i|^\alpha\right)d\boldsymbol{\xi}d\mathbf rd\mathbf s  d\mathbf u d \mathbf v,
		\end{align*}
		with $0<s_{\tau(1)}<\ldots<s_{\tau(n)}<t$ and $0<r_{\sigma(1)}<\ldots<r_{\sigma(n)}<t$.  Then following the argument leading to \eqref{eq-H-norm} in the proof of Theorem \ref{Thm-solution}, 
		we get 
		\begin{align*}
			&\frac1{n!} \E\left[\left(\Big<A_{t,x}^{\varepsilon,\delta,(j)}, A_{t,x}^{\varepsilon,\delta,(k)} \Big>_{\cH}\right)^n\right]\\
			&\le \frac{1}{n!}C^n \int_{[0,t]^{4n}}  \prod_{i=1}^n \psi_\delta(u_i-s_i)\psi_\delta(v_i-r_i)\prod_{i=1}^n|u_i-v_i|^{-\frac{d}{2q}}\\
			&\hspace{6em}  \times \prod_{i=1}^n(r_{\sigma(i+1)}-r_{\sigma(i)})^{-\frac{d}{p\alpha}}\prod_{i=1}^n(s_{\tau(i+1)}-s_{\tau(i)})^{-\frac{d}{p\alpha}}d\mathbf r d\mathbf s d\mathbf u d\mathbf v\\
			&\le \frac{1}{n!}C^n \int_{[0,t]^{2n}}  \prod_{i=1}^n(r_{\sigma(i+1)}-r_{\sigma(i)})^{-\frac{d}{p\alpha}}\prod_{i=1}^n(s_{\tau(i+1)}-s_{\tau(j)})^{-\frac{d}{p\alpha}}\prod_{j=1}^n|s_i-r_i|^{-\frac{d}{2q}}d\mathbf r d\mathbf s
		\end{align*}
		where the last step follows from  Lemma A.3 of \cite{hns11}: noting $\frac{d}{2q}\in(0,1)$, 
		\begin{align*}
			\int_{[0,t]^2}  \psi_{\delta}(u-s) \psi_{\delta}(v-r) |u-v|^{-\frac{d}{2q}} dudv
			\le C |s-r|^{-\frac{d}{2q}}.
		\end{align*} 
		Then estimate \eqref{ineq:exp-integral} can be proved by using the same argument as in the proof of Theorem \ref{Thm-solution}. The proof is concluded. 
	\end{proof}
	
	\begin{remark}
		Our proof of \eqref{e:FK-skorohod} is based on the Feynman-Kac formula for the solution and an approximation argument. In the case that the noise $\dot W$ is white time, a proof based on recursive use of It\^o’s formula was applied to obtain the Feynman-Kac formualae for the moments in  \cite{conus13}. 
	\end{remark}

	{\bf Acknowledgement} We wish to thank Davar Khoshnevisan for helpful comments.  J. Song is partially supported by the National Natural Science Foundation of China (nos. 12071256 and 12226001), the Major Basic Research Program of the Natural Science Foundation of Shandong Province in China (nos. ZR2019ZD42 and ZR2020ZD24). W. Yuan is partially supported by  ERC Consolidator Grant 815703 ``STAMFORD: Statistical Methods for High Dimensional Diffusions''.

	%\bibliographystyle{plain}
	%\bibliography{Reference-sk}

\end{document}